\documentclass[11pt, a4paper]{amsart}
\usepackage{amssymb, array, amsmath, amscd, pdfpages, enumerate, amsthm, fixltx2e, setspace, hyperref,mathtools}

\usepackage[utf8]{inputenc}
\usepackage{amssymb}
\usepackage{bm}
\usepackage{graphicx} 
\usepackage{ifthen}

\newcommand{\Dom}{D}

\newcommand{\pmat}[1]{\begin{bmatrix}#1\end{bmatrix}}
\newcommand{\pmatsmall}[1]{\begin{bsmallmatrix}#1\end{bsmallmatrix}}

\newcommand*{\keyterm}[1]{\emph{#1}}
\DeclareMathOperator{\re}{Re}	

\newcommand*{\Lin}{{\mathcal{L}}}

\newcommand{\ga}{\alpha}
\newcommand{\gb}{\beta}

\newcommand{\gd}{\delta}

\newcommand{\gl}{\lambda}

\newcommand*{\C}{{\mathbb{C}}}     
\newcommand*{\R}{{\mathbb{R}}}     

\newcommand{\citel}[2]{\cite[#2]{#1}}
\newcommand{\eq}[1]{\begin{align*}#1\end{align*}}
\newcommand{\eqn}[1]{\begin{align}#1\end{align}}

\newcommand*{\abs} [1]{\lvert#1\rvert}
\newcommand*{\norm}[1]{\lVert#1\rVert}

\newcommand*{\setm}[2]{\{\,#1\mid#2\,\}}   
\newcommand*{\iprod}[2]{\langle#1,#2\rangle}    

\newcommand*{\Abs}[2][default]{\ifthenelse{\equal{#1}{default}}{\left\lvert#2\right\rvert}{\ldelim{#1}{\lvert}#2\rdelim{#1}{\rvert}}}
\newcommand*{\Norm}[2][default]{\ifthenelse{\equal{#1}{default}}{\left\lVert#2\right\rVert}{\ldelim{#1}{\lVert}#2\rdelim{#1}{\rVert}}}
\newcommand*{\Iprod}[3][default]{\ifthenelse{\equal{#1}{default}}{\left\langle#2,#3\right\rangle}{\ldelim{#1}{\langle}#2,#3\rdelim{#1}{\rangle}}}

\newcommand*{\ddb}[2][1]{\ifthenelse{\equal{#1}{1}}{\frac{d}{d#2}}{\frac{d^{#1}}{d#2^{#1}}}}
\newcommand*{\pd}[3][1]{\ifthenelse{\equal{#1}{1}}{\frac{\partial{#2}}{\partial{#3}}}{\frac{\partial^{#1}{#2}}{\partial#3^{#1}}}}
\newcommand*{\inv}{^{-1}}
\newcommand*{\Lp}[1][p]{L^{#1}}
\newcommand*{\conj}[1]{\overline{#1}}

\newcommand{\tofrom}{\leftrightarrow}

\newcommand{\CL}{C_\Lambda}
\newcommand{\CLc}{C_{c\Lambda}}

\newtheorem{theorem}{Theorem}[section]

\newtheorem{proposition}[theorem]{Proposition}

\theoremstyle{definition}

\newtheorem{remark}[theorem]{Remark}

\numberwithin{equation}{section}

\begin{document}

\title[Polynomial Stability of Coupled PDEs]{On Polynomial Stability of Coupled Partial Differential Equations in 1D}

\thispagestyle{plain}

\author{Lassi Paunonen}
\address{Mathematics, Faculty of Information Technology and Communication Sciences, Tampere University, PO.\ Box 692, 33101 Tampere, Finland}
 \email{lassi.paunonen@tuni.fi}
 \thanks{The research is supported by the Academy of Finland Grant numbers 298182 and 310489 held by L. Paunonen.}

\begin{abstract}
  We study the well-posedness and asymptotic behaviour of selected PDE--PDE and PDE--ODE systems on one-dimensional spatial domains, namely a boundary coupled wave--heat system and a wave equation with a dynamic boundary condition. We prove well-posedness of the models and derive rational decay rates for the energy using an approach where the coupled systems are formulated as feedback interconnections of impedance passive regular linear systems.
\end{abstract}



\maketitle

\section{Introduction}

The purpose of this short paper is to discuss how an abstract ``system theoretic approach'' can be used in the study of stability properties of certain types of coupled linear PDE--PDE and PDE--ODE systems. 
In particular, several recent references have demonstrated that coupled PDE systems very often exhibit \keyterm{polynomial} and the more general \keyterm{non-uniform stability}~\cite{BatDuy08,BorTom10,RozSei19}, in which the energies of the classical solutions of the system decay at subexponential rates as $t\to\infty$.
While polynomial stability of many specific coupled PDE systems 
has been proved
in the literature~\cite{ZhaZua04,ZhaZua07,AvaLas16,BenAmm16,AmmMer12,Duy07,MerNic18} 
using a variety of powerful methods,
in this paper we focus on the usage of selected abstract results from~\cite{Pau19} establishing polynomial stability for a \keyterm{class} of such coupled systems.
More precisely, 
the results in~\cite{Pau19}  
approach the study of stability of coupled PDE--PDE and PDE--ODE systems by considering them as
abstract \keyterm{systems} which form \keyterm{a feedback interconnection}.
We demonstrate the use of this framework by proving polynomial stability for two PDE systems, 
a one-dimensional ``wave-heat system''
\begin{subequations}
  \label{eq:WHintro}
  \eqn{
    \rho(\xi)v_{tt}(\xi,t)&=(T(\xi)v_\xi(\xi,t))_\xi, \qquad -1<\xi <0, \\
    w_t(\xi,t)&=w_{\xi\xi}(\xi,t), \qquad\quad \qquad\; 0<\xi <1, \\
    v_\xi(-1,t)&=0, \qquad w(1,t)=0,\\
    v_t(0,t)&=  w(0,t), \qquad 
    T(0)v_\xi(0,t)= w_\xi(0,t) 
  }
\end{subequations}
and a wave equation with an ``acoustic boundary condition''
\begin{subequations}
\label{eq:ABCintro}
  \eqn{
    \rho(\xi) v_{tt}(\xi,t)&=(T(\xi)v_\xi(\xi,t))_\xi, \qquad 0<\xi <1, \\
     m \gd_{tt}(t) &=- d \gd_t(t)-k \gd(t)-\gb v_t(1,t)
 \\
 v_\xi(1,t)&=\gd_t(t), \qquad  
    v_t(0,t)=0. 
  }
\end{subequations}
The system~\eqref{eq:WHintro} is
similar to those studied in~\cite{ZhaZua04,BatPau16,AvaLas16}, but with a wave part that may have spatially varying density $\rho(\cdot)$ and Young's modulus $T(\cdot)$.
The system~\eqref{eq:ABCintro} is a one-dimensional analogue of wave equations used in
modelling the behaviour of acoustic waves
on higher dimensional spatial domains~\cite{Bea76,RivQin03}.

The abstract system theoretic approach has been employed in several studies on stability of coupled PDEs, especially by Ammari and co-authors~\cite{AbbAmm16,BenAmm16}, but it is still under-utilised as a technique and much of its potential remains hidden. 
This may be largely due to the fact that recognising particular PDE systems that fit in a given abstract framework is often less than straightforward, and 
formulating the particular coupled PDE system as an abstract  feedback interconnection often requires some effort. 
The purpose of this note is to demonstrate and discuss these steps for the two 
coupled systems~\eqref{eq:WHintro} and~\eqref{eq:ABCintro}.
In particular, our aim is to outline the general procedure and highlight the most important steps in using the results in~\cite{Pau19} to prove polynomial stability.

Equations~\eqref{eq:WHintro} and~\eqref{eq:ABCintro} fit into a general class of PDE--PDE and PDE--ODE systems which consist of two ``abstract component systems'' (see Figure~\ref{fig:PPI}) 
  with \keyterm{states}, \keyterm{inputs}, and \keyterm{outputs} $(x(t),u(t),y(t))$ and $(x_c(t),u_c(t),y_c(t))$
satisfying the following criteria.
\begin{itemize}
  \item[1.] One of the systems is unstable and the other is exponentially stable.
  \item[2.] Both systems are \keyterm{impedance passive}, meaning that they do not contain ``internal sources of energy'' (see Section~\ref{sec:Preliminaries} for details).
  \item[3.] The full coupled system is formed by a
``power-preserving interconnection''
\eqn{
  \label{eq:PPI}
  u(t)=y_c(t) \qquad \mbox{and} \qquad u_c(t)=-y(t)
}
(or alternatively 
$u(t)=-y_c(t)$ and $u_c(t)=y(t)$).
\end{itemize}
For such systems, the results in~\cite{Pau19} 
can be used to prove polynomial or non-uniform stability of the full coupled system by verifying certain conditions on the two component systems.

\begin{figure}[ht]
  \begin{center}
    \includegraphics[width=0.43\linewidth]{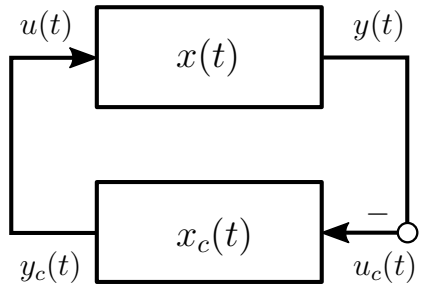}
  \end{center}
  \caption{The closed-loop system.}
  \label{fig:PPI}
\end{figure}

In many concrete PDE--PDE and PDE--ODE systems it is fairly easy to identify the unstable component 
(e.g., an undamped wave or beam equation, or an ODE with a skew-adjoint system matrix) and the stable component (e.g. a heat equation or a damped wave/beam, or a stable ODE).
However, making sure that the impedance passivity and the power-preserving interconnection are satisfied depends on the choices of the states, inputs, and outputs 
$(x(t),u(t),y(t))$ and $(x_c(t),u_c(t),y_c(t))$, as well as on the choices of the state spaces of these systems. 
Making the correct choices is often not straightforward, and this is precisely the process we aim to illustrate in this paper by considering the two PDE systems~\eqref{eq:WHintro} and~\eqref{eq:ABCintro}.
In these example cases the component systems that make up the full coupled systems have been studied extensively in the literature, and after expressing the systems as coupled abstract systems, we will find that all the conditions required for proving the polynomial closed-loop stability are either already known for our component systems, or can be computed explicitly with minimal effort.

The general framework of \keyterm{regular linear systems} used in~\cite{Pau19} involves several technical concepts. However, many of these technicalities are only required in the proofs, and
it is often not necessary to write the PDE systems under consideration in this particular form. 
Instead, the knowledge that such a representation exists is sufficient, and regularity has been proved in the literature for several particular types of PDE systems~\cite{ByrGil02,GuoSha06,GuoZha07,ZwaLeg10}.

The paper is organised as follows. In Section~\ref{sec:Preliminaries} we formulate the class of abstract systems used in the analysis in greater detail and restate a general condition for polynomial stability of abstract coupled systems from~\cite{Pau19}. 
In Sections~\ref{sec:WaveHeat} and~\ref{sec:ABC} we study the coupled wave-heat system~\eqref{eq:WHintro} and the wave equation~\eqref{eq:ABCintro}, respectively. 

If $X$ and $Y$ are Banach spaces and $A:X\rightarrow Y$ is a linear operator, we denote the domain of $A$ by $\Dom(A)$.
The space of bounded linear operators from $X$ to $Y$ is denoted by $\Lin(X,Y)$. If \mbox{$A:X\rightarrow X$,} then  $\rho(A)$ denotes
the \mbox{resolvent} set of $A$ and the resolvent operator is  \mbox{$R(\gl,A)=(\gl -A)^{-1}$}  for $\gl\in\rho(A)$.  The inner product on a Hilbert space is denoted by $\iprod{\cdot}{\cdot}$.
For $T\in \Lin(X)$ on a Hilbert space $X$ we define $\re T = \frac{1}{2}(T+T^\ast)$.
For two functions $f:[0,\infty)\to [0,\infty)$ and $g:[0,\infty)\to(0,\infty)$ we write $f(t)=o(g(t))$ if $f(t)/g(t)\to 0$ as $t\to \infty$.

\section{Coupled Abstract Systems}
\label{sec:Preliminaries}

In this section we will briefly summarise 
the most important results (in a special case related to our systems) concerning polynomial stability from~\cite{Pau19}. 
Throughout the paper we consider closed-loop system consisting of two systems with 
$(x(t),u(t),y(t))$ and $(x_c(t),u_c(t),y_c(t))$, where $x(t)\in X$ and $x_c(t)\in X_c$ for some Hilbert spaces $X$ and $X_c$, and $u(t),y(t),u_c(t),y_c(t)\in \C^m$ for all $t\geq 0$.
Most notably we assume that the two systems are \keyterm{impedance passive} in the sense that their (classical) states satisfy
\eq{
  \frac{1}{2} \ddb{t}\norm{x(t)}_X^2 \leq \re \iprod{u(t)}{y(t)}_{\C^m}
  \quad \mbox{and} \quad 
  \frac{1}{2} \ddb{t}\norm{x_c(t)}_{X_c}^2 \leq \re \iprod{u_c(t)}{y_c(t)}_{\C^m}.
}
In addition we assume that these two systems are \keyterm{regular linear systems}~\cite{Wei94} so that their dynamics are described by 
\begin{subequations}
  \label{eq:Asys}
  \eqn{
    \dot{x}(t)&= Ax(t)+Bu(t), \qquad x(0)\in X,\\
    y(t)&= \CL x(t) + Du(t)
  }
\end{subequations}
and 
\begin{subequations}
  \label{eq:Acsys}
  \eqn{
    \dot{x}_c(t)&= A_c x_c(t)+B_c u_c(t), \qquad x_c(0)\in X_c,\\
    y_c(t)&= \CLc x_c(t) + D_cu_c(t)
  }
\end{subequations}
for suitable operators $A: \Dom(A)\subset X\to X$ and $A_c : \Dom(A_c)\subset X_c\to X_c$ that generate strongly continuous semigroups, $D,D_c\in \C^{m\times m}$, and possibly unbounded operators $B,C,B_c$, and $C_c$.
The details of regular linear systems can be found, 
e.g., in~\cite{Wei94,TucWei14,Pau19}. 
While the detailed assumptions on the parameters of~\eqref{eq:Asys} and~\eqref{eq:Acsys} are fairly technical, in this paper we will demonstrate that it is often not necessary to find the exact expressions of $B,\CL,D$ and $B_c,\CLc,D_c$, or even $A$ and $A_c$, as long as the regularity of the system under consideration have been established earlier in the literature.
It should be noted that all systems on finite-dimensional spaces are regular, as are systems~\eqref{eq:Asys} with bounded input and output operators $B$ and $C$.

The impedance passivity immediately implies that the semigroups generated by $A$ and $A_c$ are contractive (this follows from letting $u(t)\equiv 0$ and $u_c(t)\equiv 0$ in the definition of impedance passivity). In addition, necessarily $\re D\geq 0$ and $\re D_c\geq 0$, where  $\re T=\frac{1}{2}(T+T^\ast)$. 

The main motivation for considering regular linear systems with possibly unbounded operators $B,\CL,B_c$, and $\CLc$ is that these systems can be used to describe
systems with couplings 
on the boundaries of the PDEs.
Another great benefit of regular linear systems is that this class has a very strong feedback theory~\cite{Wei94}. In particular, the results in~\cite{Wei94} imply that if either $D\geq 0$ or $D_c=0$, then the closed-loop system with state $x_e(t)=(x(t),x_c(t))\in X\times X_c$ is associated with a contraction semigroup $T_e(t)$ (see~\cite{Pau19} for details). 
In particular this implies that the closed-loop system has a well-defined solution and
\eq{
  \norm{x(t)}_X^2+ \norm{x_c(t)}_{X_c}^2 \leq \norm{x(0)}_X^2 + \norm{x_c(0)}_{X_c}^2, \qquad t\geq 0.
}

The results in~\cite{Pau19} establish polynomial stability of the closed-loop system under the following conditions. The theorem makes use of the \keyterm{transfer function} of the system $(x_c(t),u_c(t),y_c(t))$, which can either be computed using the formula $P_c(\gl)=\CLc (\gl-A_c)\inv B_c+D_c$ or using the Laplace transform of the original PDE system, in which case $\hat{y}_c(\gl)=P_c(\gl) \hat{u}_c(\gl)$.

\begin{proposition}
  \label{prp:CLNUStab}
  Let $(A,B,\CL,D)$ and $(A_c,B_c,\CLc,D_c)$ be two impedance passive regular linear systems where $A$ is skew-adjoint and has compact resolvent,
  and either $D\geq 0$ or $D_c=0$.
  If
  the system $(A,B,\CL,D)$ becomes exponentially stable with negative output feedback $u(t)=-y(t)$, and
 if there exists $\ga,\eta_0>0$
  such that
  \eqn{
    \label{eq:Pcbound}
    \re P_c(is)\geq \frac{\eta_0}{1+\abs{s}^\ga}
\qquad \forall s\in\R,
  }
  then the closed-loop system is polynomially stable so that 
  \eq{
    \Norm{\pmat{x(t)\\x_c(t)}}_{X\times X_c}= o(t^{-\ga}), \qquad \mbox{as} \quad t\to \infty
  }
  for all classical solutions of the closed-loop system.

    If $0\in\rho(A)$, then it is sufficient that~\eqref{eq:Pcbound} holds for $\abs{s}\geq s_0$ where $s_0>0$ is such that $[-is_0,is_0]\subset \rho(A)$.
\end{proposition}

\begin{proof}
  It is shown in~\citel{Pau19}{Thm. 3.7} that if the conditions of the proposition hold, then the generator $A_e$ of the contraction semigroup $T_e(t)$ satisfies $i\R\subset \rho(A_e)$ and $\norm{R(is,A_e)} = M_R(1+\abs{s}^\ga)$ for some constant $M_R>0$ and for all $s\in\R$.
  The claim therefore follows from~\citel{BorTom10}{Thm. 2.4}.
\end{proof}

In Proposition~\ref{prp:CLNUStab} the ``classical solutions'' of the closed-loop system refer to those solutions of the coupled PDE system that correspond to the states $x(t)$ and $x_c(t)$ for which $(x(0),x_c(0))\in X\times X_c$ belongs to the domain of the closed-loop semigroup generator.
This domain is characterised in detail in~\citel{Pau19}{Sec. 3}, but unfortunately the description in terms of the operators 
$(A,B,\CL,D)$ and $(A_c,B_c,\CLc,D_c)$ is typically not very illustrative. In this paper we will not discuss the properties of the classical solutions of the original coupled system in detail in general cases, but instead we will only present the existence of the classical solutions in the cases of the two PDE systems in Sections~\ref{sec:WaveHeat} and~\ref{sec:ABC}. As should be expected from a correct abstract formulation, these classical solutions are precisely those solutions of the original coupled systems which satisfy the boundary conditions and for which all derivatives in the system exist in a suitable sense.

\begin{remark}
  \label{rem:ObservEquivConds}
  Since we assume $A^\ast=-A$, the impedance passivity of $(A,B,\CL,D)$ 
can be used to show that the property that $u(t)=-y(t)$ stabilizes the system exponentially is equivalent to the 
\keyterm{exact observability} of the pair $(C,A)$ (or alternatively \keyterm{exact controllability} of the pair $(A,B)$)~\cite{TucWei09book,Mil12}.
\end{remark}

\section{The Coupled Wave-Heat System}
\label{sec:WaveHeat}

In this section we study the polynomial stability of the wave-heat system
\begin{subequations}
\label{eq:WH1}
  \eqn{
    \rho(\xi) v_{tt}(\xi,t)&=(T(\xi)v_\xi(\xi,t))_\xi, \qquad -1<\xi <0, \\
    w_t(\xi,t)&=w_{\xi\xi}(\xi,t), \qquad 0<\xi <1, \\
    v_\xi(-1,t)&=0, \qquad w(1,t)=0,\\
    v_t(0,t)&=  w(0,t), \qquad 
    T(0)v_\xi(0,t)= w_\xi(0,t) \label{eq:WH1coupling}
  }
\end{subequations}
with initial conditions $v(\cdot,0)\in H^1(-1,0)$, $v_t(\cdot,0)\in \Lp[2](-1,0)$, and $w(\cdot,0)\in \Lp[2](0,1)$.
Here $\rho(\cdot)$ is the mass density of the string and $T(\cdot)$ is the Young's modulus~\citel{ZwaLeg10}{Sec. 5}.
The system is similar to those considered in~\cite{ZhaZua04,BatPau16}, but the physical parameters $\rho(\cdot)$ and $T(\cdot)$ of the wave part are allowed to be spatially varying.
We assume $\rho(\cdot) $ and $T(\cdot)$ are continuously differentiable on $[-1,0]$, and $0<c_0\leq \rho(\xi),T(\xi)\leq c_1$ for some constants $c_0,c_1>0$ and for all $\xi\in [-1,0]$.

In this case the natural interpretation is that the unstable system is the wave equation on $(-1,0)$, and the stable system is the heat equation on $(0,1)$.
The coupling boundary conditions~\eqref{eq:WH1coupling} at $\xi = 0$ can indeed be interpreted as a power-preserving interconnection~\eqref{eq:PPI} if we choose
\eq{
  u_c(t)&= -w_\xi(0,t), \quad y_c(t)= w(0,t), \quad
  u(t)= v_t(0,t), \quad y(t) = T(0)v_\xi(0,t).
}
Only based on the coupling boundary conditions~\eqref{eq:WH1coupling} it would be possible to choose the converse roles for the inputs and the outputs. While this choice would also lead to a ``wave-part'' with the same properties, it turns out that the heat equation with the Dirichlet boundary input would not be a regular linear system on an $\Lp[2]$-space. Because of this, the above choice is more suitable.

Now that the inputs and outputs have been fixed, we continue by choosing a suitable state $x(t)$ and the space $X$ in such a way that the wave equation can be represented as an impedance passive regular linear system. The abstract representations of wave equations are well-understood, and 
in particular we can achieve these properties by writing the wave part 
on its ``energy space'' $X=\Lp[2](-1,0)\times \Lp[2](-1,0)$ with the  state $x(t)=(x_1(t),x_2(t))$ where
  \eq{
    x_1(t)&=\rho(\cdot)v_t(\cdot,t)  \hspace{-5ex}&&\mbox{(momentum distribution)}, \\
x_2(t)&=v_\xi(\cdot,t). \quad &&\mbox{(strain)}
  }
  In these variables the wave part becomes
    \eq{
      \ddb{t}\pmat{x_1(\xi,t)\\x_2(\xi,t)} 
&=\pmat{0&\partial_\xi\\\partial_\xi&0}
\pmat{\rho(\xi)\inv x_1(\xi,t)\\T(\xi) x_2(\xi,t)}\\
\rho(0)\inv x_1(0,t)&=u(t), \qquad  T(-1)x_2(-1,t)=0, \\
y(t)&=T(0)x_2(0,t).
    }
    If we define the norm on the state space $X=\Lp[2](-1,0)\times \Lp[2](-1,0)$ by
  \eq{
    \norm{x(t)}_X^2 
    &= \int_{-1}^0\Bigl[
    \rho(\xi)\inv\abs{x_1(\xi,t)}^2+T(\xi) \abs{x_2(\xi,t)}^2  \Bigr] d\xi,
  }
  then the total energy of the wave part is given by~\citel{Vil07phd}{Ex. 1.6}
  \eq{
    E_x(t) = \frac{1}{2}\norm{x(t)}_X^2 
  }
  for every classical solution $x(t)$.

The system operator $A$ is chosen to be
$A =   \pmatsmall{0&\partial_\xi\\\partial_\xi&0}\pmatsmall{\rho(\cdot)\inv&0\\0&T(\cdot)}$
with 
\eq{
  \Dom(A) &= \setm{(x_1,x_2)\in H^1(-1,0)\times H^1(-1,0)}{x_2(-1)=x_1(0)=0}.
}
It has been shown in~\citel{ZwaLeg10}{Sec. 5} that the wave equation is a regular linear system. In addition, it is impedance passive since every classical state $x(t)=(x_1(t),x_2(t))$ satisfies
\eq{
  \MoveEqLeft[1] \frac{1}{2}\ddb{t} \norm{x(t)}_X^2 
  = \re \iprod{\dot{x}(t)}{x(t)}_X\\
&= \re \left[\int_{-1}^0  (T(\xi)v_\xi(\xi,t))_\xi\conj{v_t(\xi,t)} 
  +  T(\xi) v_\xi(\xi,t) \conj{(v_t(\xi,t))_\xi} d\xi\right]\\
&= \re \left[ T(\xi)v_\xi(\xi,t) \conj{v_t (\xi,t)}\right]_{\xi=-1}^0
= \re \left(T(0) v_\xi(0,t) \conj{v_t (0,t)} \right)
  = \re u(t)\conj{y(t)}.
}
The system operator $A$ of the wave part is skew-adjoint and has compact resolvent by~\citel{Vil07phd}{Thm. 4.2(iv)}.
Finally, the wave part is stabilised exponentially with negative output feedback $u(t)=-\kappa y(t)$ as shown in~\citel{Vil07phd}{Ex. 5.21} (see also~\cite{CoxZua95}).

The stable part of~\eqref{eq:WH1} consisting of the heat equation is given by
\eq{
  w_t(\xi,t)&=w_{\xi\xi}(\xi,t), \qquad 0<\xi <1, \\
  w_\xi(0,t)&=-u_c(t), \qquad w(1,t)=0\\
  y_c(t)&= w(0,t).
}
This simple PDE system can be formulated as an abstract linear system on $X_c=\Lp[2](0,1)$ by choosing the state $x_c(t)=w(\cdot,t)$ and $A_c=\partial_{\xi\xi}$
with domain 
\eq{
\Dom(A_c)=\setm{x_c\in H^2(0,1)}{x_c'(0)=x_c(1)=0}.
}
The input and output operators can be chosen such that
 $B_cu_c=-\gd_0(\cdot)u_c$ for $u_c\in\C$ and
$C_cx_c = x_c(0)\in \C$ for all $ x_c(\cdot)\in \Dom(A_c)$.
With these choices we have from~\citel{TucWei14}{Prop. 6.5} that the heat system is regular.
If we choose the standard $\Lp[2]$-norm on $X_c=\Lp[2](0,1)$, then the heat system is 
also impedance passive system,
since every classical state $x_c(t)$ satisfies (using $w(1,t)=0$)
\eq{
  \frac{1}{2}\ddb{t} \norm{x_c(t)}_{\Lp[2]}^2 
  &= \re \iprod{\dot{x}_c(t)}{x_c(t)}_{\Lp[2]}
  = \re \int_0^1 w_{\xi\xi}(\xi,t)\conj{w(\xi,t)}d\xi\\
  &= \re \left[ w_\xi(\xi,t) \conj{w(\xi,t)} \right]_{\xi=0}^1 -\re \int_0^1 w_\xi(\xi,t)\conj{w_\xi(\xi,t)}d\xi\\
  &\leq  \re (-w_\xi(0,t)) \conj{w(0,t)} 
  = \re u_c(t) \conj{y_c(t)}.
}

The following proposition generalises the main results of~\cite{ZhaZua04} and~\cite{BatPau16} to the case where the wave part is allowed to have spatially varying parameters $\rho(\cdot)$ and $T(\cdot)$.

\begin{proposition}
  For all initial conditions
  \eq{
    v(\cdot,0)\in H^2(-1,0), \quad v_t(\cdot,0)\in H^1(-1,0), \quad \mbox{and}  \quad w(\cdot,0)\in H^2(0,1)
  }
  which satisfy the boundary conditions of~\eqref{eq:WH1} at $t=0$,
  the system~\eqref{eq:WH1} has a solution 
  $(v(\cdot,\cdot),w(\cdot,\cdot))$ which 
  satisfies the boundary conditions for all $t\geq 0$ and
  \eq{
    v(\cdot,\cdot)&\in C(0,\infty;H^2(-1,0))\cap C^2(0,\infty;\Lp[2](-1,0)),\\
    w(\cdot,\cdot)&\in C(0,\infty;H^2(0,1))\cap C^1(0,\infty;\Lp[2](0,1)).
  }
  The energy 
  \eq{
    E_{tot}(t) = \frac{1}{2}\int_{-1}^0 \rho(\xi)\abs{v_t(\xi,t)}^2 + T(\xi)\abs{v_\xi(\xi,t)}^2 d\xi + \frac{1}{2}\int_0^1 \abs{w(\xi,t)}^2 d\xi
  }
  of every such classical solution of~\eqref{eq:WH1} satisfies
  \eq{
    E_{tot}(t) = o(t^{-4}) .
  }
\end{proposition}

\begin{proof}
  We will not present the details in this paper, but it follows from the definition of the systems $(A,B,\CL,D)$ and $(A_c,B_c,\CLc,D_c)$ that with the given assumptions the initial state $(x(0),x_c(0))$ belongs to the domain $\Dom(A_e)$ of the generator $A_e$ of the closed-loop semigroup on $X\times X_c$. Because of this the existence and the stated properties follow from the property that the classical solution of the closed-loop system satisfies 
  $(x(t),x_c(t))\in C(0,\infty;\Dom(A_e))\cap C^1(0,\infty;X\times X_c)$. Thus
  \eq{
    v_\xi(\cdot,\cdot)&\in C(0,\infty;H^1(-1,0))\cap C^1(0,\infty;\Lp[2](-1,0))\\
    v_t(\cdot,\cdot)&\in C(0,\infty;H^1(-1,0))\cap C^1(0,\infty;\Lp[2](-1,0))\\
    w(\cdot,\cdot)&\in C(0,\infty;H^2(0,1))\cap C^1(0,\infty;\Lp[2](0,1)),
  }
  which in particular implies the first claim.

  We have $E_{tot}(t) = \frac{1}{2}\norm{x(t)}^2 + \frac{1}{2}\norm{x_c(t)}^2$, and therefore the decay rate can be deduced from Proposition~\ref{prp:CLNUStab} if we can show that the conditions are satisfied for $\ga=1/2$.
	  The system operator $A$ of the wave part is skew-adjoint and has compact resolvent by~\citel{Vil07phd}{Thm. 4.2}. 
	  It is further shown in~\citel{Vil07phd}{Ex. 5.21} that this kind of a system is stabilized with negative output feedback $u(t)=-y(t)$.
Finally, we will show that $D_c=0$ and derive lower bound of the form~\eqref{eq:Pcbound} for $\re P_c(is)$ with $\ga=1/2$.
For $\gl\in\rho(A)$ and $u_c\in\C$ we have that $P_c(\gl)u_c=y_c$ where $y_c\in \C$ is such that~\citel{CurMor09}{Sec.~1}
  \eq{
    \gl w(\xi) &= w_{\xi\xi}(\xi), \qquad \xi\in (0,1)\\
    -w_\xi(0)&=u_c, \qquad w(1)=0\\
    y_c&=w(0).
  }
  The solution $w(\xi)$ of this ODE is $w(\xi)=\frac{\sinh(\sqrt{\gl}(1-\xi))}{\sqrt{\gl}\cosh(\sqrt{\gl})}u_c$, and therefore 
  \eq{
    P_c(\gl)u_c = y_c = w(0)
    =  \frac{\sinh(\sqrt{\gl})}{\sqrt{\gl}\cosh(\sqrt{\gl})}u_c
    =  \frac{\tanh(\sqrt{\gl})}{\sqrt{\gl}}u_c.
  }
  For regular linear systems $D_c=\lim_{\gl\to \infty}P_c(\gl)$, and since $\tanh(\sqrt{\gl})$ is uniformly bounded for $\gl>0$, we have $D_c=0$.
  A direct computation also shows that 
  \eq{
    \re P_c(is)
    = \frac{1}{2\sqrt{2}\sqrt{\abs{s}}} \frac{\sinh(\sqrt{2 \abs{s}}) + \sin(\sqrt{2 \abs{s}})}{\cosh(\sqrt{2 \abs{s}})}.
  }
  Since $\re P_c(is)$ is bounded and nonzero for $-\pi/2\leq s\leq \pi/2$, and 
    $\re P_c(is)\geq 0.4 \abs{s}^{-1/2}$ for all $\abs{s}\geq \pi/2$, 
    the estimate~\eqref{eq:Pcbound} holds for $\ga=1/2$ and for some $\eta_0>0$.
  Because of this, Proposition~\ref{prp:CLNUStab} implies that for all classical solutions of the closed-loop system we have
  \eq{
E_{tot}(t)=\frac{1}{2}\Norm{\pmat{x(t)\\x_c(t)}}_{X\times X_c}^2  
= o(t^{-4}).
  }
\end{proof}

\section{Wave Equation with an Acoustic Boundary Condition}
\label{sec:ABC}

In this section we consider a one-dimensional wave equation with an ``acoustic boundary condition'', 
\begin{subequations}
\label{eq:ABC}
  \eqn{
    \rho(\xi) v_{tt}(\xi,t)&=(T(\xi)v_\xi(\xi,t))_\xi, \qquad 0<\xi <1, \\
     m \gd_{tt}(t) &=- d \gd_t(t)-k \gd(t)-\gb v_t(1,t)
 \label{eq:ABCbc1}
 \\
 v_\xi(1,t)&=\gd_t(t), \qquad  
    v_t(0,t)=0. 
 \label{eq:ABCbc2}
  }
\end{subequations}
This PDE--ODE system is
similar to those studied on multi-dimensional domains in~\cite{Bea76,RivQin03}. In particular, polynomial decay of energy was shown in the article~\cite{RivQin03} for these types of models under geometric constraints on the boundary conditions.

We again allow the physical parameters $\rho(\cdot)$ and $T(\cdot)$ to depend on the spatial variable.
The functions $\rho(\cdot)$ and $T(\cdot)$ satisfy the same assumptions as in Section~\ref{sec:WaveHeat}, and
$m>0$, $d>0$, and $k>0$ are the mass, the damping coefficient and the spring coefficient of the ODE~\eqref{eq:ABCbc1} at $\xi=1$~\cite{Bea76}.

We will prove polynomial decay of the energy of the system~\eqref{eq:ABC} by writing it as a power-preserving interconnection between two impedance passive systems --- an infinite-dimensional one and a finite-dimensional one. We begin by investigating the dynamic boundary condition~\eqref{eq:ABCbc1}. This is a second order ordinary differential equation with state $\delta(t)$, and the term $-\gb v_t(1,t)$ acts as an external input in this equation. On the other hand, the derivative $\gd_t(t)$ determines the boundary condition~\eqref{eq:ABCbc2} of the wave equation, and can therefore be considered as an output of this ODE. 
If we again consider $(x(t),u(t),y(t))$ to be the wave equation and $(x_c(t),u_c(t),y_c(t))$ to describe the ODE at $\xi=1$, then the above analysis indicates that the inputs and outputs of the component systems could be chosen as
\eq{
  y(t)&=-u_c(t)  \hspace{-10ex} && \tofrom\qquad \quad   \gb v_t(1,t) = -u_c(t)\\
  u(t)&=y_c(t)  && \tofrom \qquad\quad   v_\xi(1,t) =  \gd_t(t) 
}
However, we need to be careful in the choices of the coefficients of the inputs in order to achieve impedance passivity of the component systems.
Making of the appropriate choices is demonstrated in the following.

We can again write the wave equation on its energy space $X=\Lp[2](0,1)\times \Lp[2](0,1)$ in the variables $x_1(\xi,t)=\rho(\xi)v_t(\xi,t)$ and $x_2(\xi,t)=v_\xi (\xi,t)$ and with the norm
  \eq{
    \norm{x(t)}_X^2 
    &= \int_0^1 \Bigl[\rho(\xi)\inv\abs{x_1(\xi,t)}^2+T(\xi) \abs{x_2(\xi,t)}^2  \Bigr] d\xi.
  }
The system operator $A$ of the wave part is 
$A =   \pmatsmall{0&\partial_\xi\\\partial_\xi&0}\pmatsmall{\rho(\cdot)\inv&0\\0&T(\cdot)}$,
now with domain
\eq{
  \Dom(A) &= \setm{(x_1,x_2)\in H^1(0,1)\times H^1(0,1)}{x_1(0)=x_2(1)=0}.
}
A direct computation analogous to the one in Section~\ref{sec:WaveHeat} shows that
\eq{
  \MoveEqLeft[1.5] \frac{1}{2}\ddb{t} \norm{x(t)}_X^2 
  = \re \iprod{\dot{x}(t)}{x(t)}_X
= \re \left(T(1) v_\xi(1,t) \conj{v_t (1,t)} \right)
}
and thus in order to achieve impedance passivity for the wave part, we should choose the input and output of the wave part as $u(t)=T(1)v_\xi(1,t)$ and $y(t)=v_t(1,t)$.
These choices and the requirement for the power-preserving interconnection also fix the coefficients of the inputs and outputs of the finite-dimensional system. In particular, we have $-\gb v_t(1,t) = -\gb y(t)=:\gb u_c(t)$, and necessarily $y_c(t)=u(t)=T(1)v_\xi(1,t)=T(1)\gd_t(t)$.
The full finite-dimensional system thus becomes 
\eq{
  \ddb{t}\pmat{\gd(t)\\ \dot{\gd}(t)} &= \pmat{0&1\\-k/m&-d/m} \pmat{\gd(t)\\ \dot{\gd}(t)} + \pmat{0\\ \gb}u_c(t)\\
  y_c(t) &= \left[ 0,\; T(1) \right]\pmat{\gd(t)\\ \dot{\gd}(t)}.
}
We can define the system on $X_c=\C^2$ with matrices $(A_c,B_c,C_c,D_c)$ chosen as
\eq{
  A_c = \pmat{0&1\\-k/m&-d/m}, \quad B_c = \pmat{0\\\gb}, \quad C_c=\pmat{0,\; T(1)}, \quad D_c=0.
}
For this system the impedance passivity can be achieved with a suitable choice of the norm of  $X_c=\C^2$. Indeed, if we take a norm $\norm{(z_1,z_2)^T}_{X_c}^2=c_1 \abs{z_1}^2+c_2 \abs{z_2}^2$ for some $c_1,c_2>0$, we can compute 
\eq{
  \MoveEqLeft\frac{1}{2}\ddb{t} \norm{x_c(t)}^2_{X_c}
  = \re \iprod{\dot{x}_c(t)}{x_c(t)}_{X_c}\\
  &= c_1\re  \dot{\gd}(t)\conj{\gd(t)} 
  + c_2\re \left( -k\gd(t)/m- d \dot{\gd}(t)/m  +\gb u_c(t)\right)\conj{\dot{\gd}(t)}\\
  &\leq (c_1-c_2k/m)\re  \dot{\gd}(t)\conj{\gd(t)} 
  + \gb T(1)\inv c_2\re  u_c(t) \conj{T(1)\dot{\gd}(t)}
}
which is equal to $\re u_c(t)\conj{y_c(t)} $ if we choose $c_2=T(1)/\gb$ and $c_1=c_2 k/m$.

The following proposition establishes the polynomial decay of energy for the system~\eqref{eq:ABC}.
\begin{proposition}
	  For all initial conditions
	  \eq{
	    v(\cdot,0)\in H^2(0,1), \quad v_t(\cdot,0)\in H^1(0,1), \quad \gd(0),\gd_t(0)\in \R
	  }
	  which satisfy the boundary conditions of~\eqref{eq:ABC} at $t=0$,
	  the system~\eqref{eq:ABC} has a solution 
	  $(v(\cdot,\cdot),\gd(\cdot))$ which 
	  satisfies the boundary conditions for all $t\geq 0$ and
	  \eq{
	    v(\cdot,\cdot)\in C(0,\infty;H^2(0,1))\cap C^2(0,\infty;\Lp[2](0,1)), \qquad \gd(\cdot)\in C^2(0,\infty;\C)
	  }
	  The energy 
	  \eq{
	    E_{tot}(t) &= 
	    \frac{1}{2}\int_0^1 \rho(\xi)\abs{v_t(\xi,t)}^2+T(\xi) \abs{v_\xi(\xi,t)}^2 d\xi\\
	    &\quad +  \frac{T(1)}{2\gb m} \left( k \gd(t)^2 + m \gd_t(1,t)^2  \right)
	  }
	  of every such classical solution of~\eqref{eq:ABC} satisfies
	  \eq{
	    E_{tot}(t) = o(t^{-1}) .
	  } 
	\end{proposition}

	\begin{proof}
	  The definitions of the systems $(A,B,\CL,D)$ and $(A_c,B_c,\CLc,D_c)$ again imply that with the given assumptions the initial state $(x(0),x_c(0))$ belongs to the domain $\Dom(A_e)$ of the generator $A_e$ of the closed-loop semigroup on $X\times X_c$. Thus the closed-loop system has a classical solution such that 
	  $(x(t),x_c(t))\in C(0,\infty;\Dom(A_e))\cap C^1(0,\infty;X\times X_c)$. This also implies that the boundary conditions are satisfied for all $t\geq 0$, and
	  \eq{
    v_\xi(\cdot,\cdot)&\in C(0,\infty;H^1(-1,0))\cap C^1(0,\infty;\Lp[2](-1,0))\\
    v_t(\cdot,\cdot)&\in C(0,\infty;H^1(-1,0))\cap C^1(0,\infty;\Lp[2](-1,0)),\\
    \gd_t(\cdot,\cdot)&\in  C^1(0,\infty;\C),
	  }
which implies the first claim.

The system operator $A$ of the wave part is again skew-adjoint with compact resolvent by~\citel{Vil07phd}{Thm. 4.2} and analogously as in~\citel{Vil07phd}{Ex. 5.21} we can show that it is stabilized exponentially with negative output feedback $u(t)=-y(t)$.
Moreover, it is easy to show that $0\in\rho(A)$.

We have $D_c=0$ for the finite-dimensional system. The transfer function $P_c(is)=C_c R(is,A_c)B_c$ can be computed explicitly as
\eq{
  P_c(is) = T(1)\gb m  \cdot \frac{ds^2 + is(ks-ms^3)}{(ms^2-k)^2 + d^2s^2}, 
}
and in particular we have 
\eq{
  \re P_c(is) = T(1)\gb m \cdot \frac{ds^2 }{m^2s^4+(d^2-2km)s^2 + k^2}.
}
The transfer function is equal to zero at $s=0$, but for $\ga=2$ and for any $s_0>0$ there exists a constant $\eta_0>0$ such that the estimate~\eqref{eq:Pcbound} holds for all $\abs{s}\geq s_0$.
	  Since $E_{tot}(t) = \frac{1}{2}\norm{x(t)}_X^2 + \frac{1}{2}\norm{x_c(t)}_{X_c}^2$, the claim now follows from Proposition~\ref{prp:CLNUStab}.
	\end{proof}

\end{document}